\documentclass[psamsfonts,reqno,12pt]{amsart}
\newtheorem{theorem}{Theorem}[section]
\newtheorem{lemma}[theorem]{Lemma}
\newtheorem{proposition}[theorem]{Proposition}
\newtheorem{corollary}[theorem]{Corollary}

\theoremstyle{definition}
\newtheorem{definition}[theorem]{Definition}
\newtheorem{example}[theorem]{Example}

\theoremstyle{remark}
\newtheorem{remark}[theorem]{Remark}

\numberwithin{equation}{section}

\newcommand\N{\mathbb{N}}
\newcommand\R{\mathbb{R}}

\newcommand\K{\mathbb{K}}
\newcommand\J{\mathbb{J}}

\def\sideremark#1{\ifvmode\leavevmode\fi\vadjust{\vbox
to0pt{\vss \hbox to 0pt{\hskip\hsize\hskip1em
\vbox{\hsize2cm\tiny\raggedright\pretolerance10000
\noindent#1\hfill}\hss}\vbox to8pt{\vfil}\vss}}}

\begin{document}

\title[Schauder frames which are near-Schauder bases]
{A characterization of Schauder frames which are near-Schauder
bases}

\author{Rui Liu}

\address{Department of Mathematics and LPMC, Nankai
University, Tianjin 300071, P.R. China}

\address{Department of
Mathematics, Texas A$\&$M University, College Station, TX
77843-3368}

\email{leorui@mail.nankai.edu.cn; rliu@math.tamu.edu}

\author{Bentuo Zheng}

\address{Department of Mathematics\\ The University of Texas at Austin\\
1 University Station C1200\\ Austin, TX 78712-0257}

\email{btzheng@math.utexas.edu}



\thanks{Rui Liu was supported by funds from the Linear Analysis Workshop at Texas A$\&$M University in
2008, the
National Natural Science Foundation of China (No. 10571090), the
Doctoral Programme Foundation of Institution of Higher Education
(No. 20060055010), and the China Scholarship Council}

\thanks{Bentuo Zheng's research is supported in part by NSF grant DMS-0800061}
\begin{abstract}
A basic problem of interest in connection with the study of Schauder
frames in Banach spaces is that of characterizing those Schauder
frames which can essentially be regarded as Schauder bases. In this
paper, we give a solution to this problem using the notion of the
minimal-associated sequence spaces and the minimal-associated
reconstruction operators for Schauder frames. We prove that a
Schauder frame is a near-Schauder basis if and only if the kernel of
the minimal-associated reconstruction operator contains no copy of
$c_0$. In particular, a Schauder frame of a Banach space with no
copy of $c_0$ is a near-Schauder basis if and only if the
minimal-associated sequence space contains no copy of $c_0$. In
these cases, the minimal-associated reconstruction operator has a
finite dimensional kernel and the dimension of the kernel is exactly
the excess of the near-Schauder basis. Using these results, we make
related applications on Besselian frames and near-Riesz bases.

\end{abstract}

\date{\today}
\maketitle

\section{Introduction}

The theory of frames in Hilbert spaces presents a central tool in
mathematics and engineering, and has developed rather rapidly in the
past decade. The motivation has come from applications to signal
analysis, as well as from applications to a wide variety of areas of
mathematics, such as, sampling theory \cite{AG}, operator theory
\cite{HL}, harmonic analysis \cite{Gr}, nonlinear sparse
approximation \cite{DE}, pseudo-differential operators \cite{GH},
and quantum computing \cite{EF}. Recently, the theory of frames also
showed connections to theoretical problems such as the
Kadison-Singer Problem \cite{CFTW,CT}.

Recall that if $(x_i)$ is a standard frame for a Hilbert space $H$
with the frame transform $S$, then we get the reconstruction formula
(see \cite{Ch}):
\[x=\sum \langle x, S^{-1}x_i \rangle x_i \quad \mbox{ for all }
x\in H.\] The definition of a Schauder frame in a Banach space comes
naturally from this representation \cite{CDOSZ,CHL,HL}, which, on
the one hand, generalizes Hilbert frames, and extends the notion of
Schauder bases, on the other. Moreover, from \cite[Proposition
2.4]{CDOSZ}, the property of a Banach space $X$ to admit a
(unconditional) Schauder frame is equivalent to the property of $X$
being isomorphic to a complemented subspace of a Banach space with a
(unconditional) Schauder basis. It was shown independently by
Pe{\l}czy$\acute{\mathrm{n}}$ski \cite{Pe} and Johnson, Rosenthal
and Zippin \cite{JRZ} (see also \cite[Theorem 3.13]{Ca}) that the
property of $X$ being isomorphic to a complemented subspace of a
space with a basis is equivalent to $X$ having the Bounded
Approximation Property.

For a Schauder frame, a natural and important problem is that of
determining when it is a near-Schauder basis in the sense that the
deletion of a finite subset leaves a Schauder basis, that is, the
Schauder frame has finite excess. Section 3 is the introduction of
what we call a minimal-associated sequence space and a
minimal-associated reconstruction operator. In Section 4, we show
that a Schauder frame is a near-Schauder basis if and only if the
kernel of the minimal-associated reconstruction operator contains no
copy of $c_0$. In particular, a Schauder frame of a Banach space
with no copy of $c_0$ is a near-Schauder basis if and only if the
minimal-associated sequence space contains no copy of $c_0$. In
these cases, the minimal-associated reconstruction operator has a
finite dimensional kernel and the dimension of the kernel is exactly
the excess of the near-Schauder basis. In Section 5, we make related
applications to Besselian frames and near-Riesz bases.

\section{Preliminaries}

The unit sphere and the unit ball of a Banach space $X$ are denoted
by $S_X$ and $B_X$, respectively. The vector space of scalar
sequences $(a_i)$, which vanish eventually, is denoted by $c_{00}$.
The usual unit vector basis of $c_{00}$, as well as the unit vector
basis of $c_0$ and $\ell_p$ ($1\le p<\infty$) and the corresponding
biorthogonal functionals will be denoted by $(e_i)$ and $(e^*_i)$,
respectively. The closed linear span of a family of $(x_i)$ is
denoted by $[x_i]$.

A Schauder basis of a Banach space $X$ is a sequence
$(x_i)_{i=1}^\infty$, which has the property that every $x$ can be
uniquely written as a norm converging series $x=\sum_{i=1}^\infty
a_i x_i$. It follows then from the Uniform Boundedness Principle
that the biorthogonal functionals $(x_i^*), x_i^*:X\to \K, \sum a_j
x_j\mapsto a_i$ are bounded. Let $P_n$ be the natural projection of
$X$ onto $[x_i]_{i=1}^n$, the span of $(x_i)_{i=1}^n$. The basis
constant of $(x_i)$ is $\sup_n\{\|P_n\|\}$. The projection constant
of $(x_i)$ is $\sup_{n,m}\{\|P_n-P_m\|\}$. A family of $(x_i)$ which
is a Schauder basis of its closed linear span, $[x_i]$, is called a
basic sequence.

Two basic sequences $(x_i)$ and $(y_i)$ in the respective Banach
spaces $X$ and $Y$ are \emph{equivalent} and we write
$(x_i)\sim(y_i)$, if whenever we take a sequence of scalars $(a_i)$,
then $\sum_{i=1}^\infty a_i x_i$ converges if and only if
$\sum_{i=1}^\infty a_i y_i$ converges. Moreover, for two basic
sequences $(x_i)$ and $(y_i)$, the following conditions are
equivalent \cite{AK}:
\begin{enumerate}
\item[i)] $(x_i)\sim(y_i)$;
\item[ii)] There is an isomorphism $T:[x_i]\to [y_i]$ such that
$T(x_i)=y_i$ for each $i\in\N$;
\item[iii)] There exists a constant
$C>0$ such that for all sequences of scalars
$(a_i)\in c_{00}$ we have
\[C^{-1} \Big\|\sum a_i x_i\Big\|\le \Big\|\sum a_i y_i\Big\|\le C \Big\|\sum a_i x_i\Big\|.\]
\end{enumerate}

We say that a Banach space $X$ contains no copy of $c_0$ if $X$
contains no subspace isomorphic to $c_0$, or equivalently, there is
no basic sequence in $X$ equivalent to the unit vector basis $(e_i)$
of $c_0$. If there is a basic sequence $(x_i)$ in $X$ such that
there are positive constants $A,B>0$ such that for all finite
nonzero sequences of scalars $(a_i)\in c_{00}$ we have
\[A \max |a_i|\le \Big\|\sum a_i x_i\Big\|\le B \max |a_i|,\] then we say that $X$ contains a
$\sqrt{B/A}$-copy of $c_0$.

\section{Minimal-associated Sequence Spaces and Near-Schauder Bases}

In this section we give a short review of the concepts of Schauder
frames and associated sequence spaces (see also
\cite{CDOSZ,CHL,Liu,OSSZ}), and introduce the notion of
minimal-associated sequence spaces and near-Schauder bases.

\begin{definition}
Let $X$ be a (finite or infinite dimensional) separable Banach
space. A sequence $(x_j,f_j)_{j\in\J}$, with $(x_j)_{j\in\J}\subset
X$, $(f_j)_{j\in\J}\subset X^*$, and $\J=\N$ or $\J=\{1,2,. . . ,
N\}$, for some $N\in\N$, is called a \emph{Schauder frame of $X$} if
for every $x\in X$ \begin{equation}\label{eq:5}x=\sum_{j\in\J}
f_j(x) x_j.\end{equation} When $\J=\N$, we mean that the series in
(\ref{eq:5}) converges in norm,
$x=\displaystyle\lim_{n\to\infty}\sum_{j=1}^n f_j(x)x_j.$
\end{definition}

\begin{remark}
Throughout this paper, it will be our convention that we only
consider non-zero Schauder frames $(x_i,f_i)$ indexed by $\N$, i.e.
that $x_i\neq 0$ and $f_i\neq 0$ for $i\in\N$.
\end{remark}

\begin{definition}
Let $(x_i, f_i)$ be a Schauder frame of a Banach space $X$ and let
$E$ be a Banach space with a Schauder basis $(e_i)$. We call $(E,
(e_i))$ an \emph{associated sequence space to $(x_i, f_i)$} and
$(e_i)$ an \emph{associated Schauder basis}, if
\begin{eqnarray*}
& & S:E\to X,\quad \sum a_i e_i\mapsto \sum  a_i x_i \ \ \text{ and
} \\ & & T:X\to E,\quad x=\sum f_i(x) x_i\mapsto \sum f_i(x) e_i
\end{eqnarray*}
are bounded operators. Recall $S$ the \emph{associated
reconstruction operator} and $T$ the \emph{associated decomposition
operator} or \emph{analysis operator}.
\end{definition}

In \cite{Gr} the triple $((x_i), (f_i), E)$ is called an
\emph{atomic decomposition of $X$}.

\begin{remark}\label{rmk:1}
Notice that for all $x\in X$, \[S\circ T(x)=S\circ T\big(\sum f_i(x)
x_i\big)=S\big(\sum f_i(x) e_i\big)= \sum f_i(x) x_i=x,\] that is,
$S\circ T=\mathrm{Id}_X$. Therefore, $T$ must be an isomorphic
embedding from $X$ into $E$ and $S$ a surjection onto $X$. Moreover,
it follows easily that $E=\ker S \oplus T(X)$.
\end{remark}

\begin{definition}
Let $(x_i,f_i)$  be a Schauder frame of a Banach space $X$. We
denote the unit vector basis of $c_{00}$ by $(e_i)$ and define on
$c_{00}$ the following norm $\|\cdot\|_\mathrm{min}$:
\begin{equation}\label{eq:4}
\Big\|\sum a_i e_i\Big\|_\mathrm{min}=\max_{m\le n}
\Big\|\sum_{i=m}^n a_i x_i\Big\|_X \ \ \mbox{ for all } (a_i) \in
c_{00}.
\end{equation}
It follows easily that $(e_i)$ is a bimonotone basic sequence with
respect to $\|\cdot\|_\mathrm{min}$, which we denote by
$(\hat{e}_i)$, and thus, a Schauder basis of the completion of
$c_{00}$ with respect to $\|\cdot\|_\mathrm{min}$, which we denote
by $E_\mathrm{min}$ (see also \cite[Proposition 2.4]{CDOSZ} and
\cite[Theorem 2.6]{CHL}). From the proof of \cite[Proposition
2.4]{CDOSZ} and \cite[Theorem 2.6]{CHL}, we also know that
$(E_\mathrm{min},(\hat{e}_i))$ is an associated sequence space to
$(x_i,f_i)$. If $(E,(e_i))$ is an associated sequence space to
$(x_i,f_i)$ with $S$ the associated reconstruction operator and $K$
the projection constant of $(e_i)$, then for any $(a_i)\in c_{00}$,
\begin{eqnarray*}
\Big\|\sum a_i \hat{e}_i\Big\|&=&\max_{m\le n}\Big\|\sum_{i=m}^n a_i
x_i\Big\|=\max_{m\le n}\Big\|\sum_{i=m}^n a_i S(e_i)\Big\|\\ &\le&
\|S\| \max_{m\le n}\Big\|\sum_{i=m}^n a_i e_i\Big\|\le K \|S\|
\Big\|\sum a_i e_i\Big\|.
\end{eqnarray*}
Thus, we call $(E_\mathrm{min},(\hat{e}_i))$ the
\emph{minimal-associated sequence space to $(x_i,f_i)$} or
\emph{minimal sequence space associated to $(x_i,f_i)$} and
$(\hat{e}_i)$ the \emph{minimal-associated Schauder basis},
respectively. We call $S_\mathrm{min}:E_\mathrm{min}\to X, \sum a_i
\hat{e}_i\mapsto \sum a_i x_i$ the \emph{minimal-associated
reconstruction operator} and $T_\mathrm{min}:X\to E_\mathrm{min},
x=\sum f_i(x)x_i\mapsto \sum f_i(x)\hat{e}_i$ the
\emph{minimal-associated decomposition operator} or \emph{analysis
operator}.

\end{definition}

\begin{lemma}\label{lm:1}
$\sum_{i=1}^\infty a_i x_i$ converges in $X$ if and only if
$\sum_{i=1}^\infty a_i \hat{e}_i$ converges in $E_\mathrm{min}$.
\end{lemma}
\begin{proof}
Sufficiency is trivial by using $S_\mathrm{min}$ the
minimal-associated reconstruction operator. For necessity, if
$\sum_{i=1}^\infty a_i x_i$ converges, then for any $\epsilon>0$,
there is $N\in\N$ such that for any $N<m\le n$, $\|\sum_{i=m}^n a_i
x_i\|<\epsilon$, then \[\Big\|\sum_{i=m}^n a_i
\hat{e}_i\Big\|=\max_{m\le p\le q\le n}\Big\|\sum_{i=p}^q a_i
x_i\Big\|<\epsilon.\] It follows that $\sum_{i=1}^\infty a_i
\hat{e}_i$ converges.
\end{proof}

\begin{definition}\label{def:1}
Let $X$ a Banach space with $(x_i)\subset X$. We call $(x_i)$ a
\emph{near-Schauder basis of $X$} if there is a finite set
$\sigma\subset\N$ such that $(x_i)_{i\notin \sigma}$ is a Schauder
basis of $X$. We call a Schauder frame $(x_i,f_i)$ a near-Schauder
basis if $(x_i)$ is a near-Schauder basis.
\end{definition}

\begin{remark}
If there are two finite subsets $\sigma_1,\sigma_2\subset \N$ such
that $(x_i)_{i\notin\sigma_1}$ and $(x_i)_{i\notin\sigma_2}$ both
are Schauder bases of $X$, then $\mathrm{card} \,
\sigma_1=\mathrm{card} \, \sigma_2.$ Indeed, let
$N=\max\{i:i\in\sigma_1\cup\sigma_2\}$, then the Schauder bases
$(x_i)_{i\notin\sigma_1}$ and $(x_i)_{i\notin\sigma_2}$ both contain
the basic subsequence $(x_i)_{i> N}$. It follows easily that
$\mathrm{codim}\, [x_i]_{i>
N}=N-\mathrm{card}\,\sigma_1=N-\mathrm{card}\,\sigma_2.$ It implies
that $\mathrm{card} \, \sigma_1=\mathrm{card} \, \sigma_2.$

Then we define the \emph{excess} of a near-Schauder basis $(x_i)$ of
$X$ by
\[\mathrm{exc}(x_i)=\{\mathrm{card} \, \sigma: \exists \mbox{ a finite subset } \sigma\subset \N\, \mbox{
s.t. } (x_i)_{i\notin \sigma}
\mbox{ is a Schauder basis of } X\}.\]
\end{remark}

\section{Main Results}

The following is our main theorem, which gives a characterization of
Schauder frames which are near-Schauder bases.

\begin{theorem}\label{th:1}
Let $(x_i,f_i)$ be a Schauder frame of a Banach space $X$ and let
$(E_\mathrm{min},(\hat{e}_i))$ be the minimal-associated sequence
space to $(x_i,f_i)$ with $S_\mathrm{min}$ the minimal-associated
reconstruction operator.

Then the following conditions are equivalent:
\begin{enumerate}
\item[a)] The kernel of\,
$S_\mathrm{min}$ contains no copy of $c_0$;
\item[b)] $S_\mathrm{min}$ has a finite dimensional kernel;
\item[c)] $(x_i)$ is a near-Schauder basis of\, $X$.
\end{enumerate}
Furthermore, in this case, we have \[\mathrm{exc}(x_i)=\dim(\ker
S_\mathrm{min}).\]
\end{theorem}

Then by Theorem \ref{th:1}, we obtain the following corollary, which
gives a characterization of Schauder frames of a space with no copy
of $c_0$ which are near-Schauder bases.

\begin{corollary}\label{cor:1}
Let $(x_i,f_i)$ be a Schauder frame of a Banach space $X$ and let
$(E_\mathrm{min},(\hat{e}_i))$ be the minimal-associated sequence
space to $(x_i,f_i)$ with $S_\mathrm{min}$ the minimal-associated
reconstruction operator.

Then the following conditions are equivalent:
\begin{enumerate}
\item[a)] $E_\mathrm{min}$ contains no copy of $c_0$.
\item[b)]
\begin{enumerate}
\item[i)] $X$ contains no copy of $c_0$.
\item[ii)] $S_\mathrm{min}$ has a finite dimensional kernel.
\end{enumerate}
\item[c)]
\begin{enumerate}
\item[i)] $X$ contains no copy of $c_0$.
\item[ii)] $(x_i)$ is a near-Schauder basis of $X$.
\end{enumerate}
\end{enumerate}
Furthermore, in this case, we have\, $\mathrm{exc}(x_i)=\dim(\ker
S_\mathrm{min}).$
\end{corollary}

\begin{remark}
It is well known that reflexive spaces (or even separable dual
Banach spaces) contain no copy of $c_0$. So many classical Banach
spaces, for example $\ell_p$ $(1\leq p<\infty)$, are in this
category.
\end{remark}

For the proof of Theorem \ref{th:1} and Corollary \ref{cor:1}, we
need the following results.

\begin{proposition}\label{pp:2}
Let $(x_i,f_i)$ be a Schauder frame of a Banach space $X$ and let
$(E,(e_i))$ be an associated sequence space to $(x_i,f_i)$ with $S$
the associated reconstruction operator.

Then $S$ has a finite dimensional kernel if and only if there is a
finite subset $\sigma \subset\N$ such that $(x_i)_{i\notin\sigma}$
is a Schauder basis of $X$ which is equivalent to
$(e_i)_{i\notin\sigma}$. Furthermore, in this case, we have\,
$\mathrm{exc}(x_i)=\dim(\ker S).$
\end{proposition}
\begin{proof} Sufficiency. Suppose, to the contrary, that $\ker S$
is infinite dimensional. By hypothesis, there is a finite subset
$\sigma \subset\N$ such that $(x_i)_{i\notin\sigma}$ is a Schauder
basis of $X$ equivalent to $(e_i)_{i\notin\sigma}$. Then
$S|_{[e_i]_{i\notin\sigma}}$ is an isomorphism from
$[e_i]_{i\notin\sigma}$ onto $[x_i]_{i\notin\sigma}=X$. Since
$\dim(\ker S)=\infty$ and
$\mathrm{codim}([e_i]_{i\notin\sigma})=\mathrm{card} \sigma<\infty$,
there is $u\in [e_i]_{i\notin\sigma}\cap \ker S$ with $\|u\|=1$.
Then $S|_{[e_i]_{i\notin\sigma}}(u)=S(u)=0$, which leads to a
contradiction.

Necessity. Let $T$ be the associated decomposition operator. We
claim that $\ker S=(\mathrm{Id}_E-T\circ S)(E).$ Indeed, for any
$u\in E$, if $S(u)=0$, then $(\mathrm{Id}_E-T\circ S)(u)=u$. So
$\ker S\subset(\mathrm{Id}_E-T\circ S)(E)$. On the other hand, by
Remark \ref{rmk:1}, $S\circ T=\mathrm{Id}_X$. Then for any $u\in E$,
$S\circ(\mathrm{Id}_E-T\circ S)(u)=S(u)-S\circ T\circ
S(u)=S(u)-S(u)=0.$ It follows that $(\mathrm{Id}_E-T\circ
S)(E)\subset \ker S$. Thus, $\ker S=(\mathrm{Id}_E-T\circ S)(E)$.

Then let $Q=\mathrm{Id}_E-T\circ S$. Since $\ker S$ is finite
dimensional, $Q$ is a finite-rank operator. We claim that there is
$N\in\N$ such that \[\inf\big\{\|u-Q(u)\|:u\in [e_i]_{i\ge N},
\|u\|=1\big\}>0.\] Suppose not. That is, for all $n\in\N$,\,
$\inf\big\{\|u-Q(u)\|:u\in [e_i]_{i\ge n}, \|u\|=1\big\}=0$. Then
for any $n\in\N$, there is $u_n \in [e_i]_{i\ge n}$ so that
$\|u_n\|=1$ and $\|u_n-Q(u_n)\|<1/2^{n+1}$. Furthermore, there is
big enough $m_n\ge n$ and $\tilde{u}_n\in [e_i]_{n\le i\le m_n}$ so
that $\|\tilde{u}_n\|=1$ and
$\|\tilde{u}_n-u_n+Q(\tilde{u}_n-u_n)\|\le 1/2^{n+1}$, then
$\|\tilde{u}_n-Q(\tilde{u}_n)\|\le
\|\tilde{u}_n-u_n+Q(\tilde{u}_n-u_n)\|+\|u_n-Q(u_n)\|\le 1/2^n.$
Choose an increasing sequence $(n_i)\subset \N$ such that
$(\tilde{u}_{n_i})$ is a normalized block basic sequence of $(e_i)$,
i.e. a sequence of norm one vectors with finite increasing supports,
with $\|\tilde{u}_{n_i}-Q(\tilde{u}_{n_i})\|\le 1/2^{n_i}$ for each
$i\in\N$. Since $Q$ is a finite-rank operator and
$(Q(\tilde{u}_{n_i}))$ is a bounded sequence, there is a further
subsequence, which, without loss of generality, we still denote by
$(\tilde{u}_{n_i})$, such that $Q(\tilde{u}_{n_i})\to u_0$ for some
$u_0\in E$. It follows easily that $\|u_0\|=1$. Let $(e_i^*)\subset
E^*$ be the biorthogonal functionals of $(e_i)$. Pick $N_1$ with
$|e_{N_1}^*(u_0)|> 0$, then
\begin{eqnarray*} \mathrm{dist}(u_0,[e_i]_{i>
N_1})=\inf_{u\in [e_i]_{i> N_1}} \|u_0-u\|\ge\inf_{u\in [e_i]_{i>
N_1}}\frac{|e_{N_1}^*(u_0-u)|}{\|e_{N_1}^*\|}=\frac{|e_{N_1}^*(u_0)|}{\|e_{N_1}^*\|}>0.
\end{eqnarray*} Let $\delta=|e_{N_1}^*(u_0)|/\|e_{N_1}^*\|$. Choose $N_2>N_1$ so big that, for all
$n_i>N_2$, we
have $\|Q(\tilde{u}_{n_i})-u_0\|<\delta/2$ and $1/2^{N_2}<\delta/2.$
Take a sequence of positive numbers $(b_i)$ with
$\sum_{n_i>N_2}b_i=1.$ Then
\begin{eqnarray*} \Big\|\sum_{n_i>N_2} b_i \tilde{u}_{n_i}-
\sum_{n_i>N_2} b_i Q(\tilde{u}_{n_i})\Big\|&\ge&
\Big\|\sum_{n_i>N_2} b_i \tilde{u}_{n_i}-u_0\Big\|-
\Big\|u_0-\sum_{n_i>N_2} b_i
Q(\tilde{u}_{n_i})\Big\|\\
&\ge&\mathrm{dist}(u_0,[e_i]_{i> N_1})-\sum_{n_i>N_2} \|b_i u_0-b_i
Q(\tilde{u}_{n_i})\|\\&>&\delta - \sum_{n_i>N_2} b_i\cdot
\frac{\delta}{2}=\frac{\delta}{2}.
\end{eqnarray*}
On the other hand,
\[\Big\|\sum_{n_i>N_2} b_i \tilde{u}_{n_i}-
\sum_{n_i>N_2} b_i Q(\tilde{u}_{n_i})\Big\|\le \sum_{n_i>N_2} b_i
\|\tilde{u}_{n_i}-Q(\tilde{u}_{n_i})\|=\sum_{n_i>N_2} b_i\cdot
\frac{1}{2^{n_i}}<\frac{1}{2^{N_2}}< \frac{\delta}{2},\] which leads
to a contradiction.

Thus, there is $N\in\N$ such that $\inf\big\{\|u-Q(u)\|:u\in
[e_i]_{i\ge N}, \|u\|=1\big\}>0.$ Since $\mathrm{Id}_E-Q=T\circ S$,
these imply that $T\circ S|_{[e_i]_{i\ge N}}$ is an isomorphism from
$[e_i]_{i\ge N}$ onto $[T\circ S(e_i)]_{i\ge N}=[T(x_i)]_{i\ge N}$.
By Remark \ref{rmk:1}, $T$ is an isomorphic embedding from $X$ into
$E$, then $S|_{[e_i]_{i\ge N}}$ is an isomorphism from $[e_i]_{i\ge
N}$ onto $[x_i]_{i\ge N}$, that is, $(x_i)_{i\ge N}$ is a basic
sequence equivalent to $(e_i)_{i\ge N}$. Thus, it follows easily
that there is a finite subset $\sigma \subset\{1, . . . , N\}$ such
that $(x_i)_{i\notin\sigma}$ is a Schauder basis of $X$ equivalent
to $(e_i)_{i\notin\sigma}$.

Finally, we prove that, in this case, $\mathrm{exc}(x_i)=\dim(\ker
S)$. By hypothesis, there is a finite subset $\sigma\subset\N$ such
that $(x_i)_{i\notin\sigma}$ is a Schauder basis equivalent to
$(e_i)_{i\notin\sigma}$. Thus, it is equivalent to prove that
$\dim(\ker S)=\mathrm{card}\, \sigma=\mathrm{exc}(x_i)$. First, we
show that $\dim(\ker S)\ge\mathrm{card}\, \sigma$. It is clear that
$S|_{[e_i]_{i\notin\sigma}}$ is an isomorphism of
$[e_i]_{i\notin\sigma}$ onto $[x_i]_{i\notin\sigma}=X$. Then for any
$k\in\sigma$, since $(x_i)_{i\notin\sigma}$ is a Schauder basis of
$X$, let $(x_i^*)_{i\notin\sigma}\subset X^*$ be the biorthogonal
functionals of $(x_i)_{i\notin\sigma}$, we have
$x_k=\sum_{i\notin\sigma} x_i^*(x_k) x_i$. Thus,
\[S(e_k)=x_k=\sum_{i\notin\sigma} x_i^*(x_k) x_i=
\sum_{i\notin\sigma} x_i^*(x_k) S(e_i)=S\Big(\sum_{i\notin\sigma}
x_i^*(x_k) e_i\Big).\] Thus, $S(e_k-\sum_{i\notin\sigma} x_i^*(x_k)
e_i)=0$ for all $k\in\sigma$. $(e_k-\sum_{i\notin\sigma} x_i^*(x_k)
e_i)_{k\in\sigma}\subset \ker S$, and clearly,
$(e_k^*)_{k\in\sigma}\subset E^*$ is the orthogonal functionals. So
$(e_k-\sum_{i\notin\sigma} x_i^*(x_k) e_i)_{k\in\sigma}$ is linear
independent, it follows that $\dim(\ker S)\ge
\dim([e_k-\sum_{i\notin\sigma} x_i^*(x_k)
e_i]_{k\in\sigma})=\mathrm{card}\, \sigma$. Now, we prove that
$\dim(\ker S)\le\mathrm{card}\, \sigma$. If not, that is, $\dim(\ker
S)>\mathrm{card}\, \sigma$, then there exists finite Schauder basis
$(u_k)_{k=1}^{\dim(\ker S)}$ of $\ker S$ such that for all $1\le
k\le \dim(\ker S)$, $u_k=\sum_{i\in\sigma} e_i^*(u_k) e_i+v_k$ where
$v_k=\sum_{i\notin\sigma} e_i^*(u_k) e_i\in [e_i]_{i\notin\sigma}$.
Since $\dim(\ker S)>\mathrm{card}\, \sigma$, we know that there is a
non-zero linear combination $\sum_{k=1}^{\dim(\ker S)} \lambda_k
u_k$ of $(u_k)_{k=1}^{\dim(\ker S)}$ such that
\[\sum_{k=1}^{\dim(\ker S)} \lambda_k \sum_{i\in\sigma} e_i^*(u_k)
e_i=0.\] Thus, $\sum_{k=1}^{\dim(\ker S)} \lambda_k
u_k=\sum_{k=1}^{\dim(\ker S)} \lambda_k v_k\in
[e_i]_{i\notin\sigma}\cap\ker S=\{0\},$ which leads to a
contradiction.

\end{proof}

\begin{proposition}\label{pp:1}
Let $(x_i,f_i)$ be a Schauder frame of a Banach space $X$ and let
$(E_\mathrm{min},(\hat{e}_i))$ be the minimal-associated sequence
space to $(x_i,f_i)$ with $S_\mathrm{min}$ the minimal-associated
reconstruction operator.

If\, the kernel of $S_\mathrm{min}$ contains no copy of $c_0$, then
$\ker S_\mathrm{min}$ is finite dimensional.
\end{proposition}

\begin{proof} Suppose, to the contrary, that
$\ker S_\mathrm{min}$ is infinite dimensional. Choose $u_1\in \ker
S_\mathrm{min}$ with $\|u_1\|=1$. Then
\[0=S_\mathrm{min}(u_1)=S_\mathrm{min}\big(\sum \hat{e}_i^*(u_1)\hat{e}_i\big)
=\sum \hat{e}_i^*(u_1)S_\mathrm{min}(\hat{e}_i)=\sum
\hat{e}_i^*(u_1) x_i,\] where $(\hat{e}_i^*)\subset
E^*_\mathrm{min}$ is the biorthogonal functionals of $(\hat{e}_i)$.
Take $(\epsilon_i),(\delta_i)\subset \R^+$ with $\sum
\epsilon_i<1/2$ and $\sum \delta_i<1$. Then we can find $n_1$ so big
that $\|\sum_{i=n_1+1}^\infty
\hat{e}_i^*(u_1)\hat{e}_i\|<\epsilon_1$ and $\|\sum_{i=1}^{n_1}
\hat{e}_i^*(u_1)x_i\|<\delta_1$. Since $\dim(\ker
S_\mathrm{min})=\infty$ and $\mathrm{codim}([\,\hat{e}_i]_{i\ge
n_1+1})=n_1<\infty$, there is $u_2\in \ker S_\mathrm{min}\cap
[\,\hat{e}_i]_{i\ge n_1+1}$ with $\|u_2\|=1$. As in previous step,
from the fact that
\[0=S_\mathrm{min}(u_2)=S_\mathrm{min}\big(\sum_{i=n_1+1}^\infty
\hat{e}_i^*(u_2)\hat{e}_i\big)=\sum_{i=n_1+1}^\infty
\hat{e}_i^*(u_2)S_\mathrm{min}(\hat{e}_i)=\sum_{i=n_1+1}^\infty
\hat{e}_i^*(u_2)x_i,\] we can find $n_2>n_1$ such that
$\|\sum_{i=n_2+1}^{\infty} \hat{e}_i^*(u_2)\hat{e}_i\|<\epsilon_2$
and $\|\sum_{i=n_1+1}^{n_2} \hat{e}_i^*(u_2)x_i\|<\delta_2.$
Continuing in this fashion, we construct an increasing sequence
$\{n_i\}_{i=0}^\infty$ with $n_0=0$, and a sequence
$(u_i)_{i=1}^\infty$, where $u_i\in \ker S_\mathrm{min}\cap
[\,\hat{e}_i]_{i\ge n_{i-1}+1}$ and $\|u_i\|=1$, such that
\begin{equation*}
\Big\|\sum_{j=n_i+1}^{\infty}
\hat{e}_j^*(u_i)\hat{e}_j\Big\|<\epsilon_i \ \ \mbox{ and } \ \
\Big\|\sum_{j=n_{i-1}}^{n_i} \hat{e}_j^*(u_i)x_j\Big\|<\delta_i
\quad \mbox{ for all }\, i\in\N.
\end{equation*}
Let $\tilde{u}_i=\sum_{j=n_{i-1}+1}^{n_{i}}
\hat{e}_j^*(u_{i})\hat{e}_j$ for $i\in\N$. Then,
\[1=\|u_{i}\|\ge\Big\|\sum_{j=n_{i-1}+1}^{n_{i}}
\hat{e}_j^*(u_{i})\hat{e}_j\Big\|=\|\tilde{u}_i\|\ge
\|u_i\|-\Big\|\sum_{j=n_i+1}^{\infty}
\hat{e}_j^*(u_i)\hat{e}_j\Big\|\ge 1-\epsilon_i>\frac{1}{2}.\] It
follows that $(\tilde{u}_i)$ is a semi-normalized block basic
sequence of $(\hat{e}_i)$. Moreover,
\[\sum \|u_i-\tilde{u}_i\|=\sum \big\|\sum_{j=n_i+1}^{\infty}
\hat{e}_j^*(u_i)\hat{e}_j\big\|<\sum \epsilon_i<\frac{1}{2},\] there
is $N\in\N$ so big that $(u_i)_{i\ge N}$ is a normalized basic
sequence which is equivalent to $(\tilde{u}_i)_{i\ge N}$. Since
$(u_i)\subset \ker S_\mathrm{min}$ that contains no copy of $c_0$,
it follows that $(\tilde{u}_i)_{i\ge N}$ is not equivalent to the
unit vector basis of $c_0$. Whenever $\sum a_i \tilde{u}_i$
converges, $(a_i)\in c_0$. Thus, these imply that there must exist
$(c_i)\in c_0$ such that $\sum c_i \tilde{u}_i$ does not converges.
Let $b_j=c_i \hat{e}_j^*(u_{i+1})$ for $n_i+1\leq j\leq n_{i+1}$.

Then we claim that \[\sum_{j=1}^\infty b_j x_j=\sum_{i=0}^\infty
\sum_{j=n_i+1}^{n_{i+1}} c_i \hat{e}_j^*(u_{i+1})x_j\] converges.
Indeed, for any $\epsilon>0$, choose $N$ so big that $\sup_{i\ge N}
|c_{i}|<\min\{\frac{\epsilon}{3\|S_\mathrm{min}\|},
\frac{\epsilon}{3}\}.$ For any $n_{N}\le l\le m\in\N$, if there is
$i_0\ge 0$ such that $l,m\in [n_{i_0}+1, n_{i_0+1}]$, then
\begin{eqnarray*}\Big\|\sum_{j=l}^m b_j
x_j\Big\|&=&\Big\|\sum_{j=l}^m c_{i_0}
\hat{e}_j^*(u_{i_0+1})x_j\Big\|=|c_{i_0}|\cdot\Big\| \sum_{j=l}^m
\hat{e}_j^*(u_{i_0+1})S_\mathrm{min}(\hat{e}_j)\Big\|\\ &\le&
|c_{i_0}|\cdot \|S_\mathrm{min}\|\cdot \Big\|\sum_{j=l}^m
\hat{e}_j^*(u_{i_0+1})\hat{e}_j\Big\|\le |c_{i_0}|\cdot
\|S_\mathrm{min}\|\le \frac{\epsilon}{3}. \end{eqnarray*} If not,
there is $0\le i_1<i_2$ such that $l\in [n_{i_1}+1, n_{i_1+1}]$ and
$m\in [n_{i_2}+1, n_{i_2+1}]$, then
\begin{eqnarray*}\Big\|\sum_{j=l}^m
b_j x_j\Big\|&=&\Big\|\sum_{j=l}^{n_{i_1+1}} c_{i_1}
\hat{e}_j^*(u_{i_1+1})x_j+
\sum_{k=i_1+1}^{i_2}\sum_{j=n_{k}+1}^{n_{k+1}}
c_{k}\hat{e}_j^*(u_{k+1})x_j+ \\ & & + \sum_{j=n_{i_2}+1}^m c_{i_2}
\hat{e}_j^*(u_{i_2+1})x_j\Big\|\\
&\le& \Big\|\sum_{j=l}^{n_{i_1+1}} c_{i_1}
\hat{e}_j^*(u_{i_1+1})x_j\Big\|+ \sum_{k=i_1+1}^{i_2} |c_{k}|\cdot
\Big\| \sum_{j=n_{k}+1}^{n_{k+1}} \hat{e}_j^*(u_{k+1})x_j\Big\|+\\
& & + \Big\|\sum_{j=n_{i_2}+1}^m c_{i_2}
\hat{e}_j^*(u_{i_2+1})x_j\Big\|\\
&\le&  |c_{i_1}|\cdot \|S_\mathrm{min}\|+\sum_{k=i_1+1}^{i_2}
\delta_k \, |c_{k}|+ |c_{i_2}|\cdot \|S_\mathrm{min}\|\\
&\le&  |c_{i_1}|\cdot \|S_\mathrm{min}\|+\sup_{k>i_1}
|c_{k}|+|c_{i_2}|\cdot
\|S_\mathrm{min}\|\\&\le&\frac{\epsilon}{3}+\frac{\epsilon}{3}+\frac{\epsilon}{3}=
\epsilon.
\end{eqnarray*}
Thus, $\sum_{j=1}^\infty b_j x_j=\sum_{i=0}^\infty
\sum_{j=n_i+1}^{n_{i+1}} c_i \hat{e}_j^*(u_{i+1})x_j$ converges. So
by Lemma \ref{lm:1}, $\sum_{j=1}^\infty b_j
\hat{e}_j=\sum_{i=0}^\infty \sum_{j=n_i+1}^{n_{i+1}} c_i
\hat{e}_j^*(u_{i+1})\hat{e}_j=\sum c_i \tilde{u}_i$ converges, which
leads to a contradiction.
\end{proof}

We are now ready to present the proof of our main theorem and
corollary:

\begin{proof}[Proof of Theorem \ref{th:1}]
$(a)\Leftrightarrow(b)$ follows from Proposition \ref{pp:1}.

$(b)\Leftrightarrow(c)$ and $\mathrm{exc}(x_i)=\dim(\ker
S_\mathrm{min})$ both are obtained by Proposition \ref{pp:2}.
\end{proof}

\begin{proof}[Proof of Corollary \ref{cor:1}]
It follows from Theorem \ref{th:1}.
%
\end{proof}

The following example is a special case of our main results.

\begin{example}
Let $(z_i)$ be a normalized Schauder basis of a Banach space $X$
with biorthogonal functionals $(z_i^*)$. Let $(x_i,f_i)\subset
X\times X^*$ be defined by $x_{2k-1}=x_{2k}=z_k$ and
$f_{2k-1}=f_{2k}=z_k^*/2$ for all $k\in\N.$ Then, for every $x\in
X$, we obtain that \[x=\sum_{k=1}^\infty
z_k^*(x)z_k=\sum_{k=1}^\infty
\frac{z_k^*}{2}(x)z_k+\frac{z_k^*}{2}(x)z_k=\sum_{k=1}^\infty
f_{2k-1}(x)x_{2k-1}+f_{2k}(x)x_{2k}=\sum_{i=1}^\infty f_i(x)x_i.\]
It follows that $(x_i,f_i)$ is a Schauder frame of $X$. Let
$(E_\mathrm{min},(\hat{e}_i))$ be the minimal sequence space
associated to $(x_i,f_i)$ with $S_\mathrm{min}$ the
minimal-associated reconstruction operator. Then
\[\max_{k\in\N}|a_k|\le \Big\|\sum a_k (\hat{e}_{2k}-\hat{e}_{2k-1})\Big\|\le 2\max_{k\in\N}|a_k|\,,
\, \mbox{ for all } (a_k) \in c_{00}.\] Moreover,
$(\hat{e}_{2k}-\hat{e}_{2k-1})\subset \ker S_\mathrm{min}$. Thus,
$\ker S_\mathrm{min}$ contains a $\sqrt{2}$-copy of $c_0$.
\end{example}
\begin{proof}
First, since $S_\mathrm{min}(\hat{e}_{2k}-\hat{e}_{2k-1})
=x_{2k}-x_{2k-1}=z_k-z_k=0$, it implies that
$(\hat{e}_{2k}-\hat{e}_{2k-1})\subset \ker S_\mathrm{min}$.
Moreover, for any $n\in\N$, we have $\|\sum_{2k\le 2n} a_k
z_{k}-\sum_{2k-1\le 2n} a_k z_{k}\|=0$ and $\|\sum_{2k\le 2n-1} a_k
z_{k}-\sum_{2k-1\le 2n-1} a_k z_{k}\|=\|a_n z_n\|=|a_n|.$ Then for
any $(a_k)\in c_{00}$,
\begin{eqnarray*}
& &\Big\|\sum_{k=1}^\infty a_k
(\hat{e}_{2k}-\hat{e}_{2k-1})\Big\|=\Big\|\sum_{k=1}^\infty a_k
\hat{e}_{2k}-\sum_{k=1}^\infty a_k
\hat{e}_{2k-1}\Big\|\\&=&\max_{m\le n} \Big\|\sum_{m\le 2k\le n} a_k
x_{2k}-\sum_{m\le 2k-1\le n} a_k x_{2k-1}\Big\|\ge\max_{n\in\N}
\Big\|\sum_{2k\le n} a_k x_{2k}-\sum_{2k-1\le n} a_k
x_{2k-1}\Big\|\\&=&\max_{n\in\N} \Big\|\sum_{2k\le n} a_k
z_{k}-\sum_{2k-1\le n} a_k z_{k}\Big\|=\max_{n\in\N} |a_n|,
\end{eqnarray*}
and by the above formula, we obtain that
\begin{eqnarray*}& &\Big\|\sum_{k=1}^\infty a_k
(\hat{e}_{2k}-\hat{e}_{2k-1})\Big\|\\&=&\max_{m\le n}
\Big\|\sum_{m\le 2k\le n} a_k x_{2k}-\sum_{m\le 2k-1\le n} a_k
x_{2k-1}\Big\|\\&=&\max_{m\le n} \Big\|\sum_{2k\le n} a_k
x_{2k}-\sum_{2k\le m-1} a_k x_{2k}-\sum_{2k-1\le n} a_k
x_{2k-1}+\sum_{2k-1\le m-1} a_k x_{2k-1}\Big\|\\&\le& 2
\max_{n\in\N} \Big\|\sum_{2k\le n} a_k x_{2k}-\sum_{2k-1\le n} a_k
x_{2k-1}\Big\|=2 \max_{n\in\N} |a_n|.
\end{eqnarray*}
This completes the proof.
\end{proof}

\section{Applications to Besselian Frames and Near-Riesz Bases}

By \cite[Remark 2.7]{CDOSZ}, we know that, for a sequence $(x_j)$ in
$H$, $(x_j)$ is a Hilbert frame for $H$ if and only if there is a
sequence $(f_j)$ in $H$ such that $(x_j,f_j)$ is a Schauder frame of
$H$ and that $(\ell_2(\J), (e_j)_{j\in\J})$ (with its unit vector
basis) is an associated sequence space to $(x_j,f_j)$.

\begin{definition}\cite{Ch,Ho}
If $(x_i)_{i=1}^\infty$ is a Hilbert frame for a Hilbert space $H$,
then we say that $(x_i)$ for $H$ is
\begin{enumerate}
\item[(i)] \emph{Besselian} if whenever $\sum_{i=1}^\infty a_ix_i$ converges,
then $(a_i)\in\ell_2$;
\item[(ii)] a \emph{near-Riesz basis} if there is a finite subset $\sigma\subset
\N$ such that $(x_i)_{i\notin\sigma}$ is a Riesz basis of $H$.
\end{enumerate}
\end{definition}

Then we get the main result of \cite{Ho} as a corollary of our
theorem on a special case.
\begin{corollary}\cite{Ho}
Let $(x_i)$ be a Hilbert frame in a Hilbert space $H$.

Then the following conditions are equivalent:
\begin{enumerate}
\item[(i)] $(x_i)$ is a near-Riesz basis for $H$;
\item[(ii)] $(x_i)$ is Besselian;
\item[(iii)] $\sum_{i=1}^\infty a_i x_i$ converges in $H$ if and only if $(a_i)\in \ell_2.$
\end{enumerate}

Furthermore, in this case, we have $\mathrm{exc}(x_i)=\dim (\ker
S),$ where $S:\ell_2\to H$ with $\sum a_i e_i\mapsto \sum a_i x_i$
is the pre-frame operator.
\end{corollary}

\begin{proof} (i)$\Rightarrow$(ii)$\Rightarrow$(iii) is trivial.

(ii)$\Leftrightarrow$(i). If $(x_i)$ is Besselian, then it follows
easily that the unit vector basis of $\ell_2$ is equivalent to the
minimal-associated Schauder basis, which implies that the
minimal-associated sequence space is isomorphic to $\ell_2$, which
contains no copy of $c_0$. By Corollary \ref{cor:1}, we obtain that
$(x_i)$ is a near-Riesz basis for $H$, and $\mathrm{exc}(x_i)=\dim
(\ker S)$.
\end{proof}

\vskip1cm

\noindent\textbf{Acknowledgment.}

The authors express their appreciation to Dr. Thomas Schlumprecht
for many very helpful comments regarding frame theory in Banach
spaces. The interested readers should consult the papers
\cite{CDOSZ,CHL,Liu}. The authors also would like to thank the
referees for helpful suggestions that help us improve the
presentation of this paper.

\end{document}